\documentclass[11pt]{amsart}
\usepackage{amsmath,amssymb,amsthm,tikz}

\RequirePackage{color}
\RequirePackage[colorlinks, urlcolor=my-blue,linkcolor=my-red,citecolor=my-green]{hyperref}
\definecolor{my-blue}{rgb}{0.0,0.0,0.6}
\definecolor{my-red}{rgb}{0.5,0.0,0.0}
\definecolor{my-green}{rgb}{0.0,0.5,0.0}
\definecolor{nicos-red}{rgb}{0.75,0.0,0.0}
\definecolor{light-gray}{gray}{0.6}
\definecolor{really-light-gray}{gray}{0.8}
\definecolor{sussexg}{rgb}{0.0,0.5,0.5}
\definecolor{sussexp}{rgb}{0.5,0.0,0.5}

\usepackage{nicefrac}

\newtheorem{theorem}{\sc Theorem}[section]
\newtheorem{lemma}[theorem]{\sc Lemma}

\newtheorem{assumption}[theorem]{\bf Assumption}

\numberwithin{equation}{section}
\theoremstyle{remark}

\newtheorem{example}{\bf Example}

\newcommand{\be}{\begin{equation}}
\newcommand{\ee}{\end{equation}}

\def\bE{\mathbb{E}}
\def\bN{\mathbb{N}}
\def\bP{\mathbb{P}}

\def\N{\bN}
\def\e{\varepsilon}

\def\cS{\mathcal{S}}

\def\E{\bE}
\def\P{\bP} 

\definecolor{darkgreen}{rgb}{0.0,0.5,0.0}
\definecolor{darkblue}{rgb}{0.0,0.0,0.3}
\definecolor{nicosred}{rgb}{0.65,0.1,0.1}
\definecolor{light-gray}{gray}{0.7}
\allowdisplaybreaks[1]

\RequirePackage{datetime} 

\begin{document}
\usdate
\title
[Mixing times for semi-Markov processes]
{Bounds for mixing times for finite \\ [3pt]
   semi-Markov processes with heavy-tail \\ [3pt] jump distribution.}

\author{Nicos Georgiou}
\address{Department of Mathematics, University of Sussex, UK}
\email{n.georgiou@sussex.ac.uk}

\author{Enrico Scalas}
\address{Department of Mathematics, University of Sussex, UK}
\email{e.scalas@sussex.ac.uk}

\date{\today}
\begin{abstract}
Consider a Markov chain with finite state space and suppose you wish to change time replacing the integer step index $n$ with a random counting process $N(t)$. What happens to the mixing time of the Markov chain? We present a partial reply in a particular case of interest in which $N(t)$ is a counting renewal process with power-law distributed inter-arrival times of index $\beta$. We then focus on $\beta \in (0,1)$, leading to infinite expectation for inter-arrival times and further study the situation in which inter-arrival times follow the Mittag-Leffler distribution of order $\beta$.
\end{abstract}
\maketitle

\maketitle

\section{Introduction} \label{sec:1} 

\subsection{Motivation} 

The original motivation for this paper stems from a 2011 paper \cite{raberto11} where a probabilistic theory for dynamic networks was presented. In particular, given a fixed set of vertices, an embedded Markov chain was considered on the space of all possible graphs connecting the vertices. This discrete-time chain was then transformed into a continuous-time chain by means of a simple time change with a counting process. In a subsequent paper \cite{georgiou15}, we explicitly solved one of the models presented in \cite{raberto11}. This model is equivalent to an $\alpha$-delayed version of the Ehrenfest urn chain and the time change is the fractional Poisson process \cite{laskin03} of renewal type \cite{gorenflo04}. At that time, we initiated a discussion on how this time change for a discrete-time discrete-space Markov chain affects mixing times and the convergence rate to equilibrium. Below we collect results on this point in the interesting case in which the inter-arrival times between two consecutive transitions of the embedded chain have a power-law distribution with index $\beta$, also covering the case in which $\beta \in (0,1)$ meaning that the expected value of the waiting times is infinite. In the latter case, under an appropriate choice of the distribution of inter-arrival times, it is possible to show that the forward Kolmogorov equations can be replaced by a fractional version with Caputo derivative of index $\beta$ (see e.g. \cite{georgiou15} for details in the case mentioned above and \cite{toaldo} for a general theory) when the initial time of the process is a renewal point.

The starting point of our discussion is that the continuous-time probabilities $p_{i,j} (t)$ of being in state $j$ at time $t$, given that the process was in state $i$ at time $0$ converge to the same equilibrium distribution as in the case of the embedded chain. Then, in Theorem 2.2, we prove lower and upper bounds for the mixing time of the continuous-time chain based on the mixing time of the embedded chain and, in Theorem 2.3, we specialize the result to the case in which inter-arrival times follow the Mittag-Leffler distribution where a sharper upper bound is available. We believe these bounds can be useful for applied scientists simulating these processes, for instance to estimate how far from equilibrium their simulations are.

\subsection{Preliminaries} 

Let $T_1, T_2, \ldots$ be  a sequence of independent positive random variables with the meaning of inter-event times or waiting times (with common law $\nu$) and define the partial sum \textcolor{black}{
\begin{equation}
S_n = \sum_{k=1}^n T_i, \; \; n\ge 1.
\end{equation}
}The sequence $S_1, S_2, \ldots $ denotes the event times at
which the state of the Markov chain $X(t)$ attempts to change.

The embedded Markov chain is a discrete time chain
$X_{n}, n\ge 1$, with state space $\cS $.
Initially we assume an initial distribution $\mu^{(0)}$, i.e.~ $\P\{X_0 =i\} =\mu^{(0)}_i$ and the chain
evolves according to a discrete transition kernel $q: \cS\times \cS \to [0,1]$. As usual, since $\cS$ is finite, the transition kernel may be encoded as a transition matrix $Q = (q_{i,j})_{1 \le i,j \le |\cS|}$. We will be assuming the chain $X_n$ is {\bf irreducible} for convenience of the exposition. Otherwise, all theorems below can be ascribed to each irreducible component separately.  Moreover we shall also assume that the chain $X_n$ is {\bf aperiodic}. Again, this is a technical point when discussing the convergence to equilibrium, as in the irreducible aperiodic case we have almost sure convergence to the unique invariant measure for the discrete chain.

We couple the embedded chain $X_n$ with the process $X(t)$ via the counting process
	\be \label{eq:m-l:count}
		N_\nu(t) = \max\{ n \in \N : S_n \le t \}
	\ee
that gives the number of events from time $0$ up to a finite time horizon $t$.
Then we have
	\be\label{eq:subordinator}
		X(t) = X_{N_{\nu}(t)} = X_n 1\!\!1\{ S_n \le t < S_{n+1} \},
	\ee
i.e. the state of the process at time $t$ is the same as that
of the embedded chain after the last event before time $t$ occurred.

All information about $X(t)$ is encoded in the pairs
$\{(X_n, T_n)\}_{n \ge 1}$ which are a discrete-time Markov renewal process, satisfying
	\begin{align}\label{eq:mrp}
		\P\{ X_{n+1} = j, & T_{n+1} \le t | (X_0, S_0), \ldots, (X_n =i, S_n)\} \notag\\
				&\phantom{xxxxxx}= \P\{ X_{n+1} = j, T_{n+1}\le t | X_n = i \}.
	\end{align}
$X(\cdot)$ is then a semi-Markov process subordinated to $N_\nu(t)$ where we use ``subordination'' with the meaning of ``time change'' with an abuse of language. Under the assumption that $\mu_0 = \delta_{i}$ (deterministic starting point), the temporal evolution of its transition probabilities satisfies the forward equations
	\begin{align}
		p_{i,j}(t) \!= \overline F_\nu(t)\delta_{ij}
			\!\!+\! \sum_{\ell \in \cS} q_{\ell, j} \!\! \int_0^t \!\! p_{i,\ell}(u)f_{\nu}(t-u)\,du.
			\label{eq:master}
	\end{align}
Above we introduced $p_{i,j}(t) = \P\{X(t) = j | X(0) = i\}$, the tail (complementary cumulative distribution function)
$\overline F_\nu(t) = 1- F_\nu(t)$ and
$f_\nu(t)$ the Radon-Nikodym derivative of $\nu$ with respect to Lebesgue (the probability density function if appropriate smoothness conditions are satisfied).
These equations are proved by conditioning on the time of the last event before time $t$ and it is implicitly assumed that at $t=0$ we have a renewal point.

A conditioning argument on the values of $N_{\nu}(t)$ gives
	 \be \label{eq:anal}
	p_{i,j}(t) = \overline F_\nu(t)\delta_{ij} + \sum_{n=1}^\infty q^{(n)}_{i,j} \P\{N_\nu(t) = n\},
	\ee
	where $q^{(n)}_{i,j}$ are the $n$-step transitions of the embedded discrete Markov chain, namely  the entries
	of  the $n$-th power
	of the transition matrix $Q = (q_{i,j})_{1 \le i, j \le |\cS|}$.
	
\medskip

From the ergodic theorem we have that for any $i, j$
\[
\lim_{n \to \infty} q^{(n)}_{i,j} = \pi_{j} > 0.
\]	
This is sufficient to argue 

\begin{lemma}\label{lem:Histogram_Conv} Consider the transition probabilities given by \eqref{eq:anal} and assume that $\displaystyle \lim_{n \to \infty} q^{(n)}_{i,j} = \pi_{j}$. Then
\[
\lim_{t \to \infty} p_{i,j}(t) = \pi_i.
\]
\end{lemma}

\begin{proof}
Let $N$ large enough so that for a given $\e >0$ we have for all $n > N$
\[
|q^{(n)}_{i,j} - \pi_{i}| < \e.
\]
Then, substituting in \eqref{eq:anal} we have
\begin{align*}
p_{i,j}(t) &= \overline F_\nu(t)\delta_{ij} + \sum_{n=1}^\infty q^{(n)}_{i,j} \P\{N_\nu(t) = n\} \\
&\le  \overline F_\nu(t)\delta_{ij} + \sum_{n=1}^N q^{(n)}_{i,j} \P\{N_\nu(t) = n\} + \sum_{n=N+1}^\infty (\pi_i +\e)\P\{N_\nu(t) = n\}\\
&\le   \P\{N_\nu(t) \le N\} + (\pi_j +\e)\P\{N_\nu(t) > N\}.
\end{align*}
Then as $t \to \infty$ the first probability tends to 0, while $\P\{N_\nu(t) > N\} \to 1$. Then let $\e \to 0$ to obtain
\[
\lim_{t \to \infty} p_{i,j}(t) \le \pi_j.
\]
The lower bound follows in a similar manner so we omit details.
\end{proof} 

This straightforward convergence result is the starting point of this discussion. In discrete Markov chains, there is a substantial body of literature (see \cite{levine} and references therein) examining quantitative estimates on the convergence; this information is encapsulated in information about the mixing times of the chain, using the total variation variation distance between the two measures.

\subsection{Total variation distance and mixing times for discrete chains}	

Let $\mathcal F$ denote the $\sigma$-algebra of events of a space $\Omega$ and $\mu, \nu$ two probability measures on this space. Then the total variation distance between two measures is defined as
		\be \label{eq:tv1} \| \mu - \nu \| = \sup_{A \in \mathcal F}|\mu(A) - \nu(A)|  \in [0,1] \ee
	and one can show that for countable spaces
		\be \| \mu - \nu \| = \frac{1}{2}\sum_{x} |\mu(x) - \nu(x)|. \ee
Moreover, the total variation distance between two measures can be given in terms of a different variational formula (coupling):
	\be\label{eq:tv3}
	 \| \mu - \nu \| = \inf\{ \mathbb{P}\{ X \neq Y \}: (X, Y) \textrm{ is a coupling of } \mu \textrm{ and } \nu \}.
	\ee
Both formulas have merit, as \eqref{eq:tv1} can be used for a lower bound, while \eqref{eq:tv3} for upper bounds on mixing times.
	
	For any $\e>0$, we define the mixing time $T_{\e}$ of a finite state, aperiodic, irreducible Markov chain to be
	\be\label{210}
		T_\e = \inf\left\{ n: \sum_{ i \in \mathcal S} \| q^{(n)}_{i, \cdot} - \pi_\cdot \| \le \e\right\}.
	\ee
	The fact that  $\| q^{(n)}_{i,\cdot} - \pi_\cdot\|$ is non-increasing in $n$ means that  for all
	$N > T_\e$ we have  $\| q^{(N)}_{i, \cdot} - \pi_{\cdot} \| \le \e$ and that $T_{\e}$ are non-decreasing  as $\e \to 0$. Loosely speaking, for a given tolerance $\e$, the mixing time tells us how long it takes the chain to start behaving as if it is near equilibrium.
	
Equation \eqref{eq:tv3} can be used to obtain an upper bound for the mixing times the following way. First we construct a coupling between the two Markov chains,  where $X_0 \sim \delta_i$, $Y_0 \sim \pi$. Both chains evolve according to the transition matrix $Q$ independently, until they meet at some state $x$, after which the chains just jump to the same location together, again according to $Q$. The marginals of the pair chain $(X_n, Y_n)$ are still those of two Markov chains, so this description is indeed the description of a coupling between the two.

At the instant where the two independent chains meet, the pair Markov chain $(X_n, Y_n)$ hits the set
\[
D = \{ (x,x) : x \in \mathcal S  \}.
\]
Let the hitting time of this set be
\[
\tau_D = \inf\{ n: (X_n, Y_n)  \in D \}.
\]	
Then, using this coupling between the chains and  \eqref{eq:tv3}  one can obtain
\[
	\| q^{(n)}_{i,\cdot} - \pi\| \le \mathbb{P} \{ X_n \neq Y_n \}=\P_{\delta_i\otimes\pi}\{ \tau_D > n\}.
\]
At this point the general theory of Markov chains can assist with uniform estimates on the hitting time, irrespective of the initial measure. This can obtained by using the fact that the two chains  act independently from another- until they meet at time $\tau_{D}$- and we have
	\be \label{eq:led}
	\sup_{\mu} \P_\mu( \tau_D > n) \le c_1 e^{-c_2 n/ \ell^*_D}, \quad \ell^*_D = \max_i \E_{\delta_i}(\tau_D),
	\ee
where $c_1, c_2$ are uniform constants. In particular this gives the bound
\[
\| q^{(n)}_{i,\cdot} - \pi\| \le c_1 e^{-c_2 n/ \ell^*_D}.
\]
Using only $Q$ one can derive upper bounds for $\ell^*_D$, so we treat that as a computable constant. Now, if, overall, the upper bound is less than $\e|\cS|^{-1}$ for some $n_\e$ then \eqref{210} implies that
$T_{\e} \le n_\e$. Forcing the upper bound in the display above to be less than $\e|\cS|^{-1}$ we have that
\[
n_\e > C \ell_{D}^*(-\log \e + \log |\cS|),
\]
which in turn gives that there exists a function $f(\cS, Q)$ such that
\be \label{eq:211}
T_\e  \le f(\mathcal S, Q)| \log \e |,
\ee
which shows us how the mixing time depends on the order of $\e$.
	
For a lower bound, the most basic method involves counting; it relies on the idea that if the possible locations of a chain after $n$ jumps do not cover a substantial proportion of the state space, we cannot be close to mixing. Then one can get
		\be \label{LB}
		T_\e > \frac{\log( |\mathcal S|(1 - \e))}{\log c(Q)}\,.
		\ee
The constant $c(Q)$ only depends on the transition matrix. Note that the lower bound above is not necessarily close to the upper bound, and as $\e \to 0$ it gets weaker.  The $\e$-order of this agrees with the upper bound when $|\mathcal S| \sim \e^{-1}$. Many further methods exist for lower bounds, but are usually model-dependent.  We briefly mention that a suitable $L^2$ theory exists for reversible, aperiodic, irreducible MCs so bounds on $T_{\e}$ from below are of the same order as the upper bounds,
	\[
		((\gamma^* )^{-1}- 1) | \log 2\e | < T_{\e} < (\gamma^* )^{-1} c_{Q} |\log \e|,
	\]
where $\gamma^*$ is the spectral gap of $Q$ (the difference between $1$ and the second largest eigenvalue $\lambda_2$).

\section{Results} \label{sec:2} 

\setcounter{section}{2} \setcounter{equation}{0} \setcounter{theorem}{0} 

In this short paper, we will bound mixing times for continuous semi-Markov processes with heavy tails for the distribution of inter-event times. Using Lemma \ref{lem:Histogram_Conv} we have that the convergence occurs (albeit more slowly than Markov chains). The global time change we performed on the chain will be reflected in the bounds for the mixing times, as we obtain them in terms of the mixing times of the embedded discrete chain.

At this point, we want to  impose some conditions on the distribution of the inter-event times we are looking at. In particular 

\begin{assumption}\label{tail} 
{\sl
We assume there are two uniform constants $c_1$ and $c_2$, a $t_0 > 0$ and a $\beta>0$ such that
\[
 \frac{c_1}{t^{\beta}}\le \P\{ T > t \} \le \frac{c_2}{t^{\beta}}, \quad \textrm{ for all }\ \, t > t_0.
\]
}
\end{assumption} 

Note that we are not assuming any moments exist for the inter-event distributions as $\beta$ can be in $(0,1)$. In the case where moments exist, the results sharpen.

For any $\e>0$ we define the mixing time for the continuous semi-Markov chain to be
\be \label{eq:conmix}
T_{\e}^{\textrm{cont}} = \inf\Big\{ t:  \sum_{ i \in \mathcal S} \| p_{i, \cdot}(t) - \pi_\cdot \| \le \e \Big\}
\ee
By Lemma \ref{lem:Histogram_Conv} we know the $p_{i, \cdot}$ converge to $\pi$ so the above object is finite and well defined.

\subsection{ Motivating Examples} 

\begin{example}[Diagonalizable transition matrix] \label{Ex1} 

In this example, we make the assumption that $Q$ is the transition matrix of an irreducible, aperiodic Markov chain and in particular that it is diagonalisable.
Let $\pi$ denote the unique invariant distribution of the Markov chain and recall that $\pi$ is a left $1$-eigenvector for the matrix $Q$ and the vector $\mathbf1 = (1, \ldots, 1)$ is a right 1-eigenvector. Since $Q$ is diagonalisable, we have that there exists a matrix $L$ so that
$LQL^{-1} = D$ and without loss of generality we may assume that $d_{11} = 1$ and that $\ell_{1j}  = \pi_j$. Furthermore, by the Perron-Frobenius theorem,
the 1-eigenspace has dimension 1 and therefore the first column of $L^{-1} = (\tilde{\ell}_{ij})$ is a right 1-eigenvector of $Q$ and therefore satisfies $\tilde{\ell}_{i1} = 1$.

Then $Q^n = L^{-1}D^nL$ and on a coordinate by coordinate computation we have
$$ 
	q^{(n)}_{ij} 
=\sum_{k=1}^N \tilde\ell_{ik} \lambda^{n}_{k}\ell_{kj} 
= \pi_j + \sum_{k \neq 1} \tilde\ell_{ik} \lambda^{n}_{k}\ell_{kj}.
$$ 
The eigenvalues $\lambda_k$ remaining in the sum all have $|\lambda_k| < 1$,
with the sum vanishing as $n$ grows and the $n$-step transitions converging to the invariant distribution.

\smallskip

Substituting the last relationship back in \eqref{eq:anal}, we have
\begin{align*} \allowdisplaybreaks
	p_{i,j}(t) &= \overline F_\nu(t)\delta_{ij} + \sum_{n=1}^\infty q^{(n)}_{i,j} \P\{N_\nu(t) = n\}\\
	&= \overline F_\nu(t)\delta_{ij} + \sum_{n=1}^\infty \left( \pi_j + \sum_{k = 2}^N \tilde\ell_{ik} \lambda^{n}_{k}\ell_{kj} \right) \P\{N_\nu(t) = n\}\\
	&= \pi_j (1 - \P\{ N_\nu(t) = 0\}) +  \overline F_\nu(t)\delta_{ij}  + \sum_{k=2}^N \sum_{n=1}^\infty  \tilde\ell_{ik} \lambda^{n}_{k}  \P\{N_\nu(t) = n\}\ell_{kj}\\
	&= \pi_j + (\delta_{ij} - \pi_j) \overline F_\nu(t) + \sum_{k=2}^N \tilde\ell_{ik}  \sum_{n=1}^\infty \left( \lambda^{n}_{k} \P\{N_\nu(t) = n\} \right)\ell_{kj}\\
	&= \pi_j + \sum_{k=2}^N \tilde\ell_{ik}  \sum_{n=0}^\infty \left( \lambda^{n}_{k} \P\{N_\nu(t) = n\} \right)\ell_{kj}\\
	&= \pi_j + \sum_{k=2}^N \tilde\ell_{ik} \E \Big(\lambda_k^{N_\nu(t)}\Big)\ell_{kj}\\
	&=\pi_j + \sum_{k=2}^N \tilde\ell_{ik} P_{N_\nu(t)}(\lambda_k)\ell_{kj}.
\end{align*}

Particularly, the convergence to equilibrium for a finite state space process only depends on the tails of the probability generating function of $N_\nu(t)$. Then, since $N_\nu$ is an increasing process and $|\lambda_k| < 1$, we may bound
\be
\sup_j | p_{i,j}(t) - \pi_j| \le C_N P_{N_\nu(t)}(|\lambda_2|).
\ee
Therefore the total variation distance as a function of time only depends on the tails of the probability generating function.

In fact, the following rough estimate can be performed, keeping in mind that $|\lambda_2| <1$.  Let $K$ such that $|\lambda_2|^K < \e/2$, 
\begin{align*}
	P_{N_\nu(t)}(|\lambda_2|) &=  \E(|\lambda_2|^{N_\nu(t)}\mathbf1\{N_\nu(t) > K\}) + \E(|\lambda_2|^{N_\nu(t)}\mathbf1\{N_\nu(t) \le K\})\\
	&\le |\lambda_2|^K + \P\{ N_{\nu}(t) \le K\} \le \e/2 +  \P\{ N_{\nu}(t) \le K\} .
\end{align*}

 From Lemma \ref{lm:tails} 
 below, the second term above decays like $K^{1+\beta}t^{-\beta}$ and modulates $c$ in order to make this quantity arbitrarily small.
\end{example}

\begin{example} \label{Ex2} 
(Mittag-Leffler waiting times)
This example is taken from \cite{lambiotte}.
When the waiting times $T_i$ are Mittag-Leffler with parameter $\beta$, we have that
$P_{N_\nu(t)}(\lambda) = E_{\beta}((\lambda-1)t^\beta)$ where $E_{\beta}$ is the Mittag-Leffler function with parameter $\beta \in (0,1]$.
For large $t$ values we have that
\[
E_{\beta}((\lambda-1)t^\beta) \sim C_{\lambda, \beta} t^{-\beta},
\]
and therefore
\be
\sup_j | p_{i,j}(t) - \pi_j| \le C_{\lambda_2, \beta} N t^{-\beta}.
\ee
The total variation distance becomes less than $\e > 0$, when
\[
t > \left(\frac{C_{\lambda, \beta, N}N}{\e}\right)^{1/\beta}.
\]
We compute an explicit value for $C_{\lambda, \beta, N}$ later, in the proof of Theorem \ref{mixingtimesML}.
\end{example}

We are now ready to state the main theorem.

\begin{theorem} \label{mixingtimes} 
Assume \ref{tail}. Let $\e > 0$ and  $T^{\rm{emb}}_{\e/2}$ the $\e/2$-mixing time for the embedded chain, given by \eqref{210}. Then for any $\beta > 0 $ we can find explicit constants $C_1$ so that
\[
 T^{\rm{cont}}_{\e} < C_1 \e^{-1/\beta} (T^{\rm{emb}}_{\e/2})^{1+1/\beta}.
\]
\end{theorem}

This theorem is quite general as it makes no further assumptions on the background chain. Moreover, as is often the case for discrete Markov chains, a lot of the sophisticated estimates on mixing times are model dependent, so a theorem like \ref{mixingtimes} can utilise those bounds directly.

\smallskip

In the case where the inter-event times are Mittag-Leffler distributed we can make the upper bound sharper.

\begin{theorem}\label{mixingtimesML} 
Let $X(t)$ be a finite space semi-Markov process for which the inter-event times are Mittag-Leffler$(\beta)$ distributed. Then,
\[
T^{\rm{cont}}_{\e} < C_2 \e^{-1/\beta} (T^{\rm{emb}}_{\e/2})^{1/\beta}\,.
\]
\end{theorem}

In the figure below, we see a simulation of the fractional Ehrenfest chain for times before and at the upper bound of the mixing time in Theorem \ref{mixingtimesML}.

\begin{figure}[h] 
 \includegraphics[height=4.3cm]{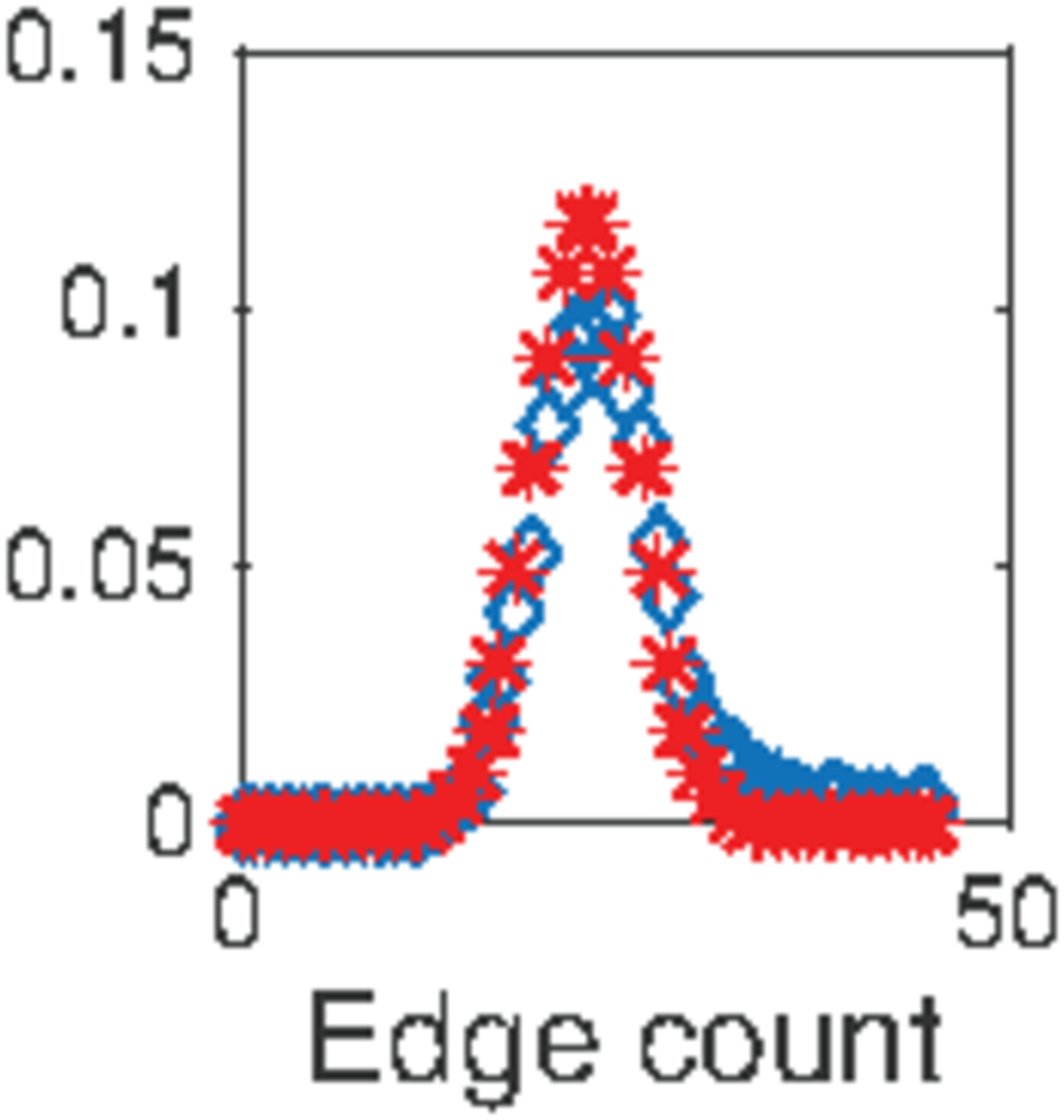} 
 \hspace{1.5cm} 
  \includegraphics[height=4.3 cm]{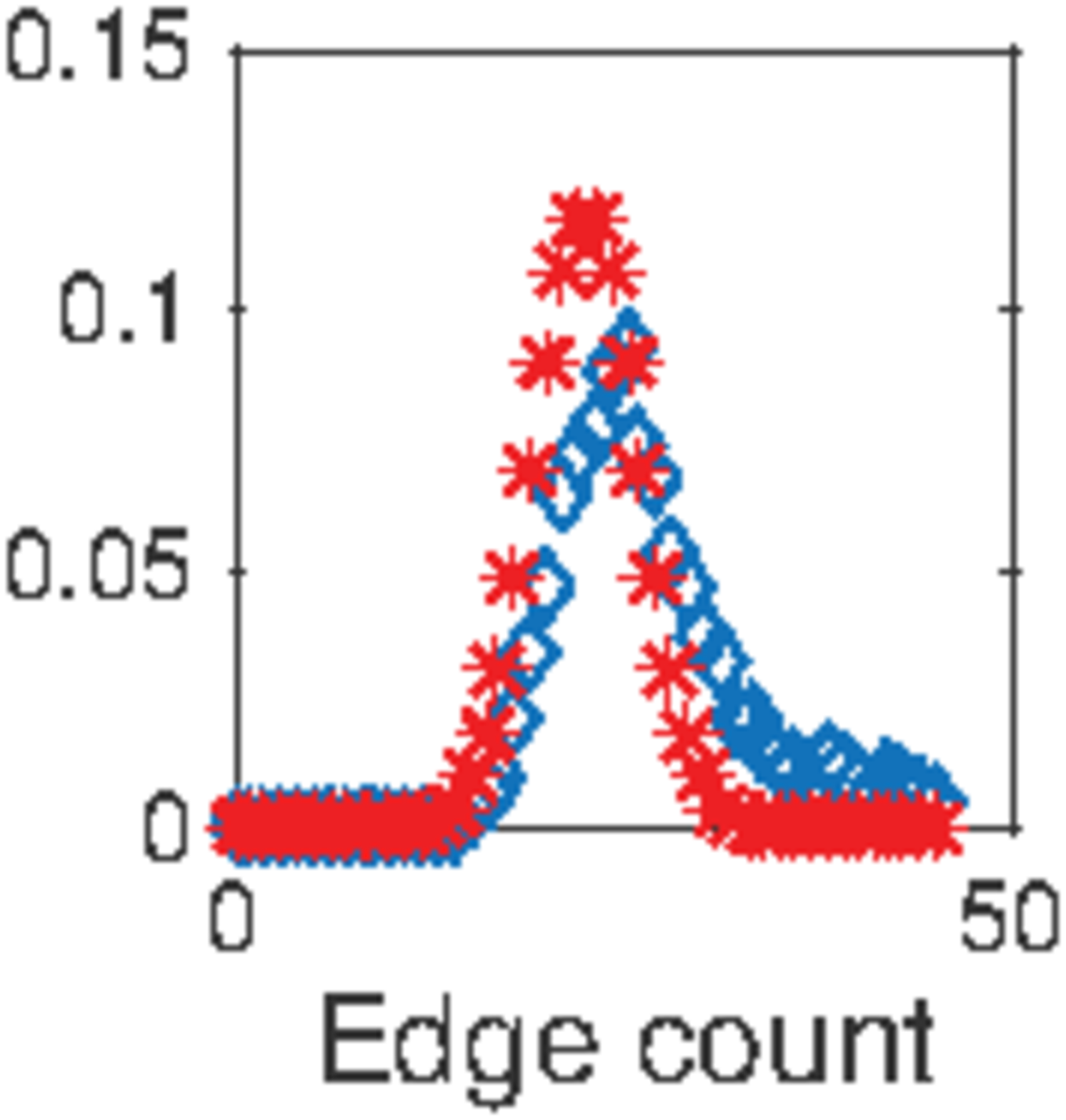} 
  \caption{Mittag-Leffler Ehrenfest chain at (left panel) and before (right panel) the bound for $T^{\rm{cont}}_{\e}$ as given in Theorem \ref{mixingtimesML}. The abscissae represent states and the ordinates are their empirical frequencies (blue stars). The expected distribution is binomial and is represented by the red stars.}
\end{figure}

A natural question arises about lower bounds for $T^{\rm{cont}}_{\e}$. These are more challenging to obtain for the total variation distance directly. However, by defining a different distance between the measures we can obtain also lower bounds. 
Let 
\be \label{eq:newdef}
\widetilde{T}^{\rm{cont}}_\e = \inf\{ t: \max_{i} \E\| q^{(N_s)}_{i,\cdot} - \pi \| < \e, \,\, \text{ for all } s> t \}.
\ee
Note that 
\[
\| p_{i, \cdot}(t) - \pi \| = \| \E(q^{(N_t)}_{i, \cdot}) - \pi \| \le   \E\|q^{(N_t)}_{i, \cdot}- \pi \|,
\]
and therefore if the expected value \eqref{eq:newdef} is less than $\e$ then the total variation distance is small. In particular this already gives
\[
\widetilde{T}^{\rm{cont}}_\e >{T}^{\rm{cont}}_\e.
\]

Using definition \eqref{eq:newdef}, we can however find bounds for $\widetilde{T}^{\rm{cont}}_\e$.

\begin{theorem} \label{mixingtimesND}
Assume \ref{tail}. Let $\delta > 0$ and  $T^{\rm{emb}}_{\delta}$ the $\delta$-mixing time for the embedded chain, given by \eqref{210}. Then for any $\beta > 0 $, and any $\alpha \in (0,1)$ we can find explicit constants $C_1 < C_2$ so that 
\[
C_1 \e^{(\alpha-1)/\beta} (T^{\rm{emb}}_{\e^{\alpha}})^{1/\beta} < \widetilde{T}^{\rm{cont}}_{\e} < C_2 \e^{-1/\beta} (T^{\rm{emb}}_{\e/2})^{1+1/\beta}.
\]
\end{theorem}
We are now ready to present the proofs in the next section.

\section{Mixing times and equilibrium}\label{sec:3} 

\setcounter{section}{3} \setcounter{equation}{0} \setcounter{theorem}{0} 

\begin{lemma} \label{lm:tails}
Under Assumption \ref{tail}, let  $K \in \N$ and let $t > (t_0 \vee (2c_2)^{1/\beta})K$. Then, there exists a uniform positive constant $C_0$ so that
\be\label{eq:tails}
\frac{c_1 \,K}{t^\beta} - \frac{C_0\,K^{2}}{t^{2\beta}} \le  \P\{N_\beta(t)  < K \}  \le \frac{c_2K^{1+\beta}}{t^\beta} + \frac{C_0\,K^{1+2\beta}}{t^{2\beta}}.
\ee
\end{lemma}

\allowdisplaybreaks

\begin{proof} 
The assumptions of the lemma guarantee that all functions below are well defined, all constants arising from Taylor's theorem do not depend on $t$ and the error of Taylor's theorem is small.
When $t > (t_0 \vee (2c_2)^{1/\beta})K$ we have
\begin{align}
1 - \P\{N_\beta(t)  < K \} &= 1 - \P\left\{ t < \sum^{K}_{j=1} T_j \right\}  \ge 1- \P\left\{ t <   K \max_{1 \le j\le K } T_j \right\}\notag \\
& = 1-  \P\left\{ \max_{1 \le j\le K} T_j  > \frac{t}{K} \right\}
= \left(\P\left\{ T_1  \le  \frac{t}{K} \right\}\right)^{K} \notag \\
&\ge \exp\left\{K\log \left( 1 -  c_2 \frac{K^\beta}{t^\beta} \right)\right\} \notag\\
&= \exp\left\{-Kc_2 \frac{K^\beta}{t^\beta} - KC_{\textrm{up}}c^2_2 \frac{K^{2\beta}}{t^{2\beta}} \right\},
\,  \textrm{for a uniform $C_{\textrm{up}}$,} \notag \\
&\ge 1 - c_2 \frac{K^{1+\beta}}{t^\beta} - C_{\textrm{up}}c^2_2 \frac{K^{1+2\beta}}{t^{2\beta}} \label{eq:tailrhs}.
\end{align}

For a lower bound we can write
\begin{align}
\P\{N_\beta(t)  < K \} &= \P\left\{ t < \sum^{K}_{j=1} T_j \right\} \ge  \P\left\{ t <  \max_{1 \le j\le K } T_j \right\} \notag\\
&= 1 - \left(\P\left\{ T_1  < t \right\}\right)^{K} \ge 1- \left(1 - \frac{c_1}{t^{\beta}}\right)^K \notag\\
&= 1 - \exp\left\{-Kc_1 \frac{1}{t^\beta} - K \widetilde C_{\textrm{low}}c^2_1 \frac{1}{t^{2\beta}} \right\}\ge 1 - \exp\left\{-Kc_1 \frac{1}{t^\beta}\right\} \notag \\
&\ge Kc_1 \frac{1}{t^\beta} - C_{\textrm{\textrm{low}}}\left(\frac{K c_1}{t^\beta}\right)^2,
\  \textrm{ for a uniform constant $C_{\textrm{low}}$.} \label{eq:taillhs}
\end{align}
The lemma follows from \eqref{eq:tailrhs} and \eqref{eq:taillhs}. The last inequality on the right side of \eqref{eq:tails} comes directly from the assumption.
\end{proof} 

\begin{proof}[Proof of Theorem \ref{mixingtimes}]

It suffices to prove that for arbitrary $M< L$ the total variation distance between the transition probabilities  and the equilibrium distribution
is bounded above, according to the following
	\be \label{eq:tvpi}
	 \| p_{i,\cdot}(t) - \pi\| \le \P\{ N_t < M\} + \|q^{(M)}_{i,\cdot} - \pi\| + \| q_{i,\cdot}^{(L)} - \pi \|\P\{ N_t > L\}.
	 \ee
Assume for the moment that  \eqref{eq:tvpi} holds and set $M = T^{\textrm{emb}}_{\e/2}$. By the definition of $T^{\textrm{emb}}_{\e/2}$, the middle term on the right-hand side of \eqref{eq:tvpi} is bounded above by $\e/2$. Then let $L \to \infty$ so that the third term vanishes.

\smallskip

The left-hand side is then bounded by $\e$ -and therefore the continuous process is  $\e$-mixed- if $\P\{ N_t < T^{\textrm{emb}}_{\e/2} \} \le \e/2$. By Lemma \ref{lm:tails}
this happens whenever
\[ t >  \left(\frac{2(c_1 + C_0)}{\e}\right)^{1/\beta}\left( T^{\textrm{emb}}_{\e/2}\right)^{1+1/\beta} \vee (t_0 \vee (2c_2)^{1/\beta})T^{\textrm{emb}}_{\e/2},\]
and therefore
	\be
		T^{\textrm{cont}}_{\e} < C_2 \e^{-1/\beta} (T^{\textrm{emb}}_{\e/2})^{1+1/\beta}.
	\ee
	
	 The theorem is proven when we establish \eqref{eq:tvpi}. To this end,
	 \allowdisplaybreaks
\begin{align*}
& 2\|p_{i,\cdot}(t) - \pi \| = \sum_{j \in \mathcal S}|p_{i,j}(t) - \pi_{j}| \\
	&=  \sum_{j \in \mathcal S}|\sum_{n=0}^{\infty}(q^{(n)}_{i,j} - \pi_{j}) \P\{ N_t = n\}|
	\le \sum_{j \in \mathcal S}\sum_{n=0}^{\infty}|q^{(n)}_{i,j} - \pi_{j}| \P\{ N_t = n\}\\
	&\le \P\{ N_t < M\} + \sum_{j \in \mathcal S}\sum_{n=M}^{L}|q^{(n)}_{i,j} - \pi_{j}| \P\{ N_t = n\} +  2\| q^{(L)}_{i,\cdot} - \pi \|\P\{ N_t > L\}\\
	&\le \P\{ N_t < M\} +2 \|q^{(M)}_{i,\cdot} - \pi\| \P\{ M \le N_t \le L\}
	+ 2\| q^{(L)}_{i, \cdot} - \pi \|\P\{ N_t > L\}. 
\end{align*}

\end{proof} 

\begin{proof}[Proof of Theorem \ref{mixingtimesML}] 

When the counting process $N_{\beta}(t)$ has Mittag-Leffler($\beta$) inter-event times, we have
\be
\bar n_t = \E(N_\beta(t)) = \frac{t^\beta}{\Gamma(1+\beta)}, \quad  \E(N^2_\beta(t)) =  \bar n_t + (\bar n_t)^2
\left[ \frac{\beta B(\beta,1/2)}{2^{2 \beta -1}} -1 \right],
\ee
where $B(\cdot,\cdot)$ is the beta function.

As in Example \ref{Ex2}, 
\begin{align}
p_{i,j}(t) &= \overline F_\nu(t)\delta_{ij} + \sum_{n=1}^\infty q^{(n)}_{i,j} \P\{N_\nu(t) = n\} = \sum_{n=0}^\infty q^{(n)}_{i,j} \P\{N_\nu(t) = n\} \notag\\
&\le \sum_{n=0}^\infty (\pi_j + c_1e^{-n/e\ell^*_D}) \P\{N_\nu(t) = n\} \notag\\
&=\pi_j + c_1 M_{N_\nu(t)}(-1/e\ell^*_D) =\pi_j + c_1 E_{\beta}((e^{-1/e\ell^*_D}-1)t^\beta) \label{eq:MLmix}.
\end{align}
Here, $M_{N_\beta(t)}(s)$ is the moment generating function of the counting process $N_{\beta}(t)$, while $E_{\beta}$ is the Mittag-Leffler function with parameter $\beta$. $\ell^*_D$ is defined in equation \eqref{eq:led}.

We will extrapolate mixing times asymptotics by forcing
\[
c_1E_{\beta}((e^{-1/e\ell^*_D}-1)t^\beta) = c_1 M_{N_\beta(t)}(-1/e\ell^*_D) < \e.
\]
The equality between this two quantities is a beautiful fact of the Mittag-Leffler function. The derivation of the moment generating function can be found in the book \cite{baleanu} and in \cite{laskin03}.

One way to bound above the moment generating function is by
\be \label{mgb}
M_{N_\nu(t)}(-1/e\ell^*_D) \le \P\left\{ N_{\beta}(t) \le \theta\frac{t^{\beta}}{\Gamma(1+\beta)} \right\}+ e^{- \theta t^{\beta}/e\Gamma(1+\beta)\ell^*_D}.
\ee
The constant $\theta$ is to be determined so that each term above is bounded by $\e/2$. For the first term we will use the Paley-Zygmound inequality. For any
$\theta \in [0,1]$ we have
\begin{align*}
 \P\left\{ N_{\beta}(t) \ge \theta\frac{t^{\beta}}{\Gamma(1+\beta)} \right\} &\ge  \frac{(1 - \theta)^2 \E(N_\beta(t))^2}{\textrm{Var}(N_\beta(t)) + (1 - \theta)^2\E(N^2_\beta(t))}\\
 &=\frac{ (1 - \theta)^2 \bar n_t^2}{((1-\theta)^2 + 1)\bar n_t + \bar n_t^2 \left(\frac{\beta B(\beta, 1/2)}{2^{2\beta-1}} -1 \right)}\\
 &=\frac{1}{ (1 - \theta)^{-2} \left(\frac{\beta B(\beta, 1/2)}{2^{2\beta-1}} -1 \right) + (1 +  (1 - \theta)^{-2})\bar n_t^{-1}}.
 \end{align*}
The function $\frac{\beta B(\beta, 1/2)}{2^{2\beta-1}} -1$ is monotonically decreasing in $\beta$ and takes values in $(0,1)$. therefore there is a unique $\theta^*(\beta)$ in $(0,1)$ so that  $(1-\theta^*(\beta))^{-2}(\frac{\beta B(\beta, 1/2)}{2^{2\beta-1}} -1)= 1$. For the particular value of $\theta^*$ we bound
\begin{align*}
 \P\Big\{ N_{\beta}(t) &\ge \frac{\theta^*({\beta})t^{\beta}}{\Gamma(1+\beta)} \Big\} \\
 &\ge 1 - \frac{(1 + (1 - \theta^*(\beta))^{-2})}{\bar{n_t}} = 1 -  \frac{(1 + (1 - \theta^*(\beta))^{-2})\Gamma(1+\beta)}{t^{\beta}}.
\end{align*}
In particular we obtain
\be\label{49}
 \P\left\{ N_{\beta}(t) \le \frac{\theta^*(\beta)t^{\beta}}{\Gamma(1+\beta)} \right\} \le  \frac{(1 + (1 - \theta^*(\beta))^{-2})\Gamma(1+\beta)}{t^{\beta}} = \frac{C_\beta}{t^{\beta}}.
\ee
This is a much improved bound for the probability, than the one established in Lemma \ref{lm:tails}.
Impose that the upper bound in \eqref{49} is less than $\e /2$ to obtain that
\be\label{eq:p1}
t > \left(\frac{2C_\beta}{\e}\right)^{1/\beta}.
\ee
Similarly, set
\be \label{eq:p2}
e^{- \theta^*(\beta) t^{\beta}/e\Gamma(1+\beta)\ell^*_D} < \e/2 \Longleftrightarrow t > \left( \frac{e\Gamma(1+\beta)}{\theta^*(\beta)}\ell_D^*\log\frac{2}{\e}\right)^{1/\beta}.
\ee
Combine \eqref{eq:p1} and \eqref{eq:p2} in \eqref{mgb}, which in turn can bound \eqref{eq:MLmix} to conclude that the relation
\be
T^{\rm{cont}}_\e \le  \left(\max\left\{ C_\beta,  \frac{e\Gamma(1+\beta)}{\theta^*(\beta)}\ell_D^*\right\} \frac{2c_1}{\e}\right)^{1/\beta},
\ee
as required.
\end{proof} 

\begin{proof}[Proof of Theorem \ref{mixingtimesND}]
Using definition \eqref{eq:newdef}, we can however find a lower bound for $\widetilde{T}^{\rm{cont}}_\e$.	
We have that for any $M$ positive, 
\be
 \E\|q^{(N_t)}_{i, \cdot}- \pi \| \ge  \E(\|q^{(N_t)}_{i, \cdot}- \pi \| {\bf 1}\{ N_t < M \}) \ge \|q^{(M)}_{i, \cdot}- \pi \|\P\{ N_t < M \}.
\ee
If we set $M = \frac{1}{2}T^{\textrm{emb}}_{\e^\alpha}$, we have 
\[
 \E\|q^{(N_t)}_{i, \cdot}- \pi \| \ge \e^{\alpha}\P\Big\{ N_t < \frac{1}{2}T^{\textrm{emb}}_{\e^\alpha} \Big\}, 
\]
and therefore it suffices to have $\P\{ N_t < T^{\textrm{emb}}_{\e^\alpha}/2 \} > \e^{1-\alpha}$, in order for the two measures to not be close in distance \eqref{eq:newdef}.
This is enough to guarantee 
	 \[
	  \widetilde{T}^{\textrm{cont}}_{\e}  \ge  \sup\Big\{ t:  \P\Big\{ N_{t} < \frac{1}{2}T^{\textrm{emb}}_{\e^{\alpha}}\Big\} \ge \e^{1-\alpha}\Big\}.  
	 \]
	 At this point we need to separate two cases, depending on the assumption of Lemma \ref{lm:tails}. 
	 If $\beta < 1$, then the assumption of the lemma requires 
	\[
	 t > C(t_0, c_2, \beta) T^{\textrm{emb}}_{\e^{\alpha}} 
	 \]
	 in order to use \eqref{eq:tails}, while we must also have  
	  \be \label{P0} 
	 t^{\beta} > C(t_0, c_2, \beta) T^{\textrm{emb}}_{\e^{\alpha}}, 
	 \ee
	 so that the lower bound in Lemma \ref{lm:tails} is non-negative.
	Then
	\be \label{P2}
	\e^{1-\alpha} < \tilde{C}(t_0, c_2, \beta) \frac{T^{\textrm{emb}}_{\e^{\alpha}}}{t^\beta} \Longleftrightarrow t <  \tilde{C}(t_0, c_2, \beta) \e^{(\alpha-1)/\beta}(T^{\textrm{emb}}_{\e^{\alpha}})^{1/\beta}.
	\ee
	In order for both inequalities \eqref{P0} and \eqref{P2} to be satisfied, we need (modulo the constants)
	\[
	1 <  \e^{(\alpha-1)/\beta} 
	\]
	which is true as $\alpha <1$. Therefore in the case $\beta<1$
	 \[
	 \widetilde{T}^{\textrm{cont}}_{\e^{\alpha}} > C_1 \e^{(\alpha-1)/\beta} (T^{\textrm{emb}}_{\e^{\alpha}})^{1/\beta}.
	 \]
Now suppose that $\beta \ge 1$. Then for the estimate in Lemma \ref{lm:tails} to be meaningful (i.e. the lower bound is strictly greater than 0), we need that $t^{\beta}> T^{\textrm{emb}}_{\e^{\alpha}}$. 
This is guaranteed by the assumption of Lemma \ref{lm:tails} and we obtain 
	 \[
	\widetilde{T}^{\textrm{cont}}_{\e^{\alpha}} \stackrel{\textrm{Lemma \ref{lm:tails}}}{>}  C_1 \e^{(\alpha-1)/\beta} (T^{\textrm{emb}}_{\e^{\alpha}})^{1/\beta}.\]

Now for the upper bound in the theorem, we can repeat the arguments of Theorem \ref{mixingtimes}. We have 
\begin{align*}
\E\| q^{(N_t)}_{i, \cdot} - \pi \| = \sum_{n=0}^{\infty}\| q^{(n)}_{i, \cdot} - \pi \|\P\{N_t = n\},
\end{align*}
and therefore bound \eqref{eq:tvpi} and all subsequent arguments work for this distance as well. 
\end{proof}

\section*{Acknowledgements}

Both authors were partially supported by the {\em Dr Perry James (Jim) Browne Research Center} at the
Department of Mathematics, University of Sussex.

\bigskip \smallskip

 \it




\begin{thebibliography}{99}\normalsize 

\bibitem{baleanu}
D. Baleanu, K. Diethelm, E. Scalas, J.J. Trujillo, {\em Fractional Calculus: Models and Numerical Methods}. World Scientific, Singapore (2016).

\bibitem{lambiotte}
S. de Nigris, A. Hastir, R. Lambiotte, Burstiness and fractional diffusion on complex networks. {\em Eur. Phys. J. B}
{\bf 89} (2016), \# 114.

\bibitem{georgiou15}
N. Georgiou, I.Z. Kiss, E. Scalas, Solvable non-Markovian dynamic network. {\em Phys. Rev. E}  {\bf 92} (2015), \# 042801.

\bibitem{gorenflo04} F. Mainardi, R. Gorenflo, E. Scalas,
A fractional generalization of the Poisson process. {\em Vietnam J. Math.} {\bf 32}, SI (2004), 53--64.

\bibitem{laskin03}
N. Laskin, Fractional Poisson process. {\em Commun. Nonlinear Sci, Numer. Simul.} {\bf 8} (2003), 201--213.

\bibitem{toaldo}
M.M. Meerschaert, B. Toaldo, Relaxation patterns and semi-Markov dynamics.
{\em Stoch. Process. Their Appl.} {\bf 129} No 8 (2019), 2850--2879.

\bibitem{levine}
D.A. Levin, Y. Peres, E.L. Wilmer, {\em Markov Chains and Mixing Times}. American Mathematical Society (2009).

\bibitem{raberto11}
M. Raberto, F. Rapallo, E. Scalas, Semi-Markov graph dynamics. {\em PLOS ONE} {\bf 6} No 8 (2011), \# e23370.

\end{thebibliography}
\end{document}